\documentclass[10pt]{amsart}
\usepackage{latexsym}
\usepackage{amsfonts,amsmath,amssymb,indentfirst}
\newcommand{\beq}{\begin{equation}}
\newcommand{\eeq}{\end{equation}}
\newcommand{\bdism}{\begin{displaymath}}
\newcommand{\edism}{\end{displaymath}}

\newcommand{\al}{\alpha}
\newcommand{\be}{\beta}
\newcommand{\ga}{\gamma}
\newcommand{\de}{\delta}
\newcommand{\alp}{\overline{\alpha}}
\newcommand{\bet}{\overline{\beta}}
\newcommand{\gam}{\overline{\gamma}}
\newcommand{\del}{\overline{\delta}}
\newcommand{\et}{\eta}

\newcommand{\setZ}{\mathbb Z}

\newcommand{\setR}{\mathbb R}
\newcommand{\setC}{\mathbb C}

\newtheorem{main}{Theorem}

\newtheorem{others}{Theorem}
\newtheorem{theorem}{Theorem}[section]
\newtheorem{proposition}[theorem]{Proposition}
\newtheorem{corollary}[theorem]{Corollary}

\newtheorem{definition}{Definition}

\author[Di Cerbo]{Luca Fabrizio Di Cerbo}

\title[Finite-volume complex-hyperbolic surfaces]{Finite-volume complex-hyperbolic surfaces, their toroidal
compactifications, and geometric applications}
\begin{document}
\maketitle

\begin{abstract}
We study the classification of smooth toroidal compactifications of
nonuniform ball quotients in the sense of Kodaira and Enriques.
Moreover, several results concerning the Riemannian and complex
algebraic geometry of these spaces are given. In particular we show
that there are compact complex surfaces which admit Riemannian
metrics of nonpositive curvature, but which do not admit K\"ahler
metrics of nonpositive curvature. An infinite class of such examples
arise as smooth toroidal compactifications of ball quotients.
\end{abstract}

\section{Introduction}
\pagenumbering{arabic}

Let $\tilde{M}$ be a symmetric space of noncompact type, and let
$\textrm{Iso}_{0}(\tilde{M})$ denote the connected component of the
isometry group of $\tilde{M}$ containing the identity. Recall that
$\textrm{Iso}_{0}(\tilde{M})$ is a semi-simple Lie group. A discrete
subgroup $\Gamma\subset \textrm{Iso}_{0}(\tilde{M})$ is a
\emph{lattice} in $\tilde{M}$ if $\tilde{M}/\Gamma$ is of finite
volume. When $\Gamma$ is torsion free, then $\tilde{M}/\Gamma$ is a
finite volume manifold or a locally symmetric space. A lattice
$\Gamma$ is \emph{uniform} (\emph{nonuniform}) if $\tilde{M}/\Gamma$
is compact (noncompact).

The theory of compactifications of locally symmetric spaces or
varieties has been extensively studied, see for example
\cite{Borel}. In fact, locally symmetric varieties of noncompact
type often occur as moduli space in algebraic geometry and number
theory, see \cite{Mumford}. For technical reasons this beautiful
theory is mainly developed for quotients of symmetric spaces or
varieties by \emph{arithmetic} subgroups. For arithmetic subgroups
of semi-simple Lie groups a nice reduction theory is available
\cite{Borel}. Among many other things, the aforementioned theory can
be used to deduce their finite generation, the existence of finitely
many conjugacy classes of maximal parabolic subgroups, and the
existence of \emph{neat} subgroups of finite index.

The celebrated work of Margulis \cite{Margulis} implies that
lattices in any semi-simple Lie group of real rank bigger or equal
than two are arithmetic subgroups. This important theorem does not
cover many interesting cases such as lattices in the complex
hyperbolic space $\setC\mathcal{H}^{n}$, where non-arithmetic
lattices are known to exist by the work of Mostow and
Mostow-Deligne; see \cite{Deligne} and the bibliography therein.

It is thus desirable to develop a theory of compactifications of
locally symmetric varieties modeled on $\setC\mathcal{H}^{n}$
regardless of the arithmeticity of the defining torsion free
lattices. A compactification of finite-volume complex-hyperbolic
manifolds as a complex spaces with isolated normal singularities was
obtained by Siu and Yau in \cite{Siu}. This compactification may be
regarded as a generalization of the Baily-Borel compactification
defined for arithmetic lattices in $\setC\mathcal{H}^{n}$. A
\emph{toroidal} compactification for finite-volume
complex-hyperbolic manifolds was described by Hummel and Schroeder
in connection with cusps closing techniques arising from Riemannian
geometry \cite{Schroeder}; see also the preprint by Mok \cite{Mok}
and the classical reference \cite{Mumford} for what concerns the
arithmetic case.

The constructions of both Siu-Yau and Hummel-Schroeder rely on the
theory of nonpositively curved Riemannian manifolds. The key point
here is that the structure theorems for finite-volume manifolds of
negatively pinched curvature, or more generally for
\emph{visibility} manifolds \cite{Eberlein}, can be used as a
substitute of the reduction theory for arithmetic subgroups.

In this paper we study torsion-free nonuniform lattices in the
complex hyperbolic plane $\setC\mathcal{H}^{2}$ and their toroidal
compactifications. Let $\Gamma$ be a lattice as above and let
$\overline{\setC\mathcal{H}^{2}/\Gamma}$ denote its toroidal
compactification. When $\overline{\setC\mathcal{H}^{2}/\Gamma}$ is
smooth, it is a compact K\"ahler surface \cite{Hummel}. It is then
of interest to place these smooth K\"ahler surfaces in the framework
of the \emph{Kodaira-Enriques} classification of complex surfaces
\cite{Van de Ven}. The main purpose of this paper is to prove the
following:

\begin{main} \label{jubilo} Let $\Gamma$ be a nonuniform torsion-free
lattice in $\setC\mathcal{H}^{2}$. There exists a finite subset
$\mathcal{F}^{'}\subset\Gamma$ of parabolic isometries for which the
following holds: for any normal subgroup $\Gamma^{'}\lhd\Gamma$ with
the property that $\mathcal{F}^{'}\cap\Gamma^{'}$ is empty, then
$\overline{\setC\mathcal{H}^{2}/\Gamma^{'}}$ is a surface of general
type with ample canonical line bundle. Moreover,
$\overline{\setC\mathcal{H}^{2}/\Gamma^{'}}$ admits Riemannian
metrics of nonpositive sectional curvature but it cannot support
K\"ahler metrics of nonpositive sectional curvature.
\end{main}

An outline of the paper follows. Section $\textrm{II}$ starts with a
summary of the results of Hummel and Schroeder \cite{Hummel}. Such
results are then combined with the Kodaira-Enriques classification
to prove that when the lattice $\Gamma$ is sufficiently small then
$\overline{\setC\mathcal{H}^{2}/\Gamma}$ is a surface of general
type with ample canonical bundle.

In section $\textrm{III}$ we present some examples of a surfaces of
general type which do not admit any nonpositively curved K\"ahler
metric, but whose underlying smooth manifolds admit Riemannian
metrics of nonpositive curvature. Finally the proof of Theorem A is
given.

In section $\textrm{IV}$ we show how Theorem \ref{jubilo}, combined
with the theory of semi-stable curves on algebraic surfaces
\cite{Sakai}, can be used to address the problem of the
projective-algebraicity of minimal compactifications (Siu-Yau) of
finite-volume complex-hyperbolic surfaces. The results of section
$\textrm{IV}$ are then summarized in Theorem \ref{algebraicity}. The
result obtained is effective.

The projective-algebraicity of minimal compactifications is proved,
through $L^{2}$-estimates for the $\overline{\partial}$-operator, by
Mok in \cite{Mok}. This analytical approach works in any dimension.

\section{Toroidal Compactifications and the Kodaira-Enriques Classification}

Let $\textrm{PU}(1,2)$ denote the connected component of
$\textrm{Iso}(\setC\mathcal{H}^{2})$ containing the identity. Let
$\Gamma$ be a nonuniform torsion-free lattice of holomorphic
isometries of the complex hyperbolic plane $\setC\mathcal{H}^{2}$,
i.e., $\Gamma\leq \textrm{PU}(1,2)$. Recall that the locally
symmetric space $\setC\mathcal{H}^{2}/\Gamma$ has finitely many cusp
ends $A_{1}, ..., A_{n}$ which are in one to one correspondence with
conjugacy classes of the maximal parabolic subgroups of $\Gamma$
\cite{Eberlein1}. The set of all parabolic elements of $\Gamma$ can
be written as a disjoint union of subsets $\Gamma_{x}$, where
$\Gamma_{x}$ is the set of all parabolic elements in $\Gamma$ having
$x$ as unique fixed point. Here $x$ is a point in the natural point
set compactification of $\setC\mathcal{H}^{2}$ obtained by adjoining
points at infinity corresponding to asymptotic geodesic rays. Thus,
given a cusp $A_{i}$, let us consider the associated maximal
parabolic subgroup $\Gamma_{x_{i}}\leq \Gamma$ and the horoball
$\textrm{HB}_{x_{i}}$ stabilized by $\Gamma_{x_{i}}$. We then have
that $\textrm{HB}_{x_{i}}/\Gamma_{x_{i}}$ is naturally identified
with $A_{i}$.

Recall that after choosing an Iwasawa decomposition \cite{Eberlein}
for $\textrm{PU}(1,2)$, we get a identification of
$\partial\textrm{HB}$ with the three dimensional Heisenberg Lie
group $N$. Moreover, $N$ comes equipped with a left invariant metric
and then we may view $\Gamma_{x_{i}}$ as a lattice in
$\textrm{Iso}(N)$. The cusps $A_{1}, ..., A_{n}$ are then identified
with $N/\Gamma_{x_{i}}\times [0,\infty)$, for $i=1, ...,n$.

The isometry group of $N$ is isomorphic to the semi-direct product
$\textrm{Iso}(N)=N\rtimes U(1)$. We say that a lattice in
$\textrm{Iso}(N)$ is rotation free if it is a lattice in $N$, i.e.,
if it is a lattice of left translations. A parabolic isometry
$\phi\in\Gamma$ is called \emph{unipotent} if it acts as a
translation on its invariant horospheres.

We now briefly summarize some of the results of Hummel \cite{Hummel}
and Hummel-Scroeder \cite{Schroeder}.

\begin{others}[Hummel-Scroeder]\label{Hummel1}
Let $\Gamma$ be a nonuniform torsion-free lattice in
$\setC\mathcal{H}^{2}$. Then, there exists a finite subset
$\mathcal{F}\subset\Gamma$ of parabolic isometries such that for any
normal subgroup $\Gamma^{'}\lhd\Gamma$ with the property that
$\mathcal{F}\cap\Gamma^{'}$ is empty, then
$\overline{\setC\mathcal{H}^{2}/\Gamma^{'}}$ is smooth and K\"ahler.
\end{others}

Furthermore, using a cusp closing technique arising from Riemannian
Geometry they were able to prove:

\begin{others}[Hummel-Schroeder]\label{Hummel2}
Let $\Gamma$ be a nonuniform torsion-free lattice in
$\setC\mathcal{H}^{2}$. Then, there exists a finite subset
$\mathcal{F^{'}}\subset\Gamma$ of parabolic isometries such that
$\mathcal{F^{'}}\supseteq\mathcal{F}$ for which the following holds.
For any normal subgroup $\Gamma^{'}\lhd\Gamma$ with the property
that $\mathcal{F^{'}}\cap\Gamma^{'}$ is empty, then
$\overline{\setC\mathcal{H}^{2}/\Gamma^{'}}$ admits a Riemannian
metric of nonpositive sectional curvature.
\end{others}

A few remarks about these results. A nonuniform torsion-free lattice
in $\setC\mathcal{H}^{2}$ admits a smooth toroidal compactification
if its parabolic isometries are all unipotent. In the arithmetic
case this is achieved by choosing a neat subgroup of finite index
\cite{Mumford}. It is also interesting to observe that we have
plenty of normal subgroups satisfying the requirements of Theorem
\ref{Hummel1} and \ref{Hummel2}, in fact $\textrm{PU}(1,2)$ is
linear and then \emph{residually finite} by a fundamental result of
Mal'cev \cite{Malcev}. Finally, it is interesting to notice that in
general one expects the strict inclusion
$\mathcal{F^{'}}\supset\mathcal{F}$ to hold. Explicit examples can
be derived from the construction of Hirzebruch \cite{Hirzebruch}.

For simplicity, a compactification as in Theorem \ref{Hummel2} will
be referred as toroidal \emph{Hummel-Schroeder} compactification.

\begin{proposition}\label{Milnor}
Let $M$ be a finite-volume complex-hyperbolic surface which admits a
toroidal Hummel-Schroeder compactification. Then the Euler number of
$\overline{M}$ is strictly positive.
\end{proposition}

\begin{proof}

The idea for the proof goes back to an unpublished result of J.
Milnor about the Euler number of closed four dimensional Riemannian
manifolds having sectional curvatures along perpendicular planes of
the same sign; see the paper by S. S. Chern \cite{Chern}. Let
$(\overline{M},g)$ be the Riemannian manifold obtained by closing
the cusps of $M$ under the condition of nonpositive curvature
\cite{Schroeder}. Let $\Omega$ be its curvature matrix. We can
always choose \cite{Chern} a orthonormal frame $\{e_{i}\}^{4}_{i=1}$
such that:
\begin{align}\notag
R_{1231}=R_{1241}=R_{1232}=R_{1242}=R_{1332}=R_{1341}=0.
\end{align}
It follows that
\begin{align}\notag
Pf(\Omega)&=\Omega^{1}_{2}\wedge\Omega^{3}_{4}-\Omega^{1}_{3}\wedge\Omega^{2}_{4}+\Omega^{1}_{4}\wedge\Omega^{2}_{3}\\
\notag
&=\{R_{1221}R_{3443}+R_{1243}^{2}+R_{1331}R_{2442}+R_{1342}^{2}\\
\notag &+R_{1441}R_{2332}+R_{1234}^{2}\}d\mu_{g},
\end{align}
where $Pf(\Omega)$ is the Pfaffian of the skew symmetric matrix
$\Omega$. The statement is now a consequence of Chern-Weil theory.

\end{proof}

We can now use the Kodaira-Enriques classification of closed smooth
surfaces \cite{Van de Ven} to derive the following theorem. The
proof is in the spirit of the theory of nonpositively curved spaces.

\begin{theorem}\label{luca}
Let $M$ be a finite-volume complex-hyperbolic surface which admits a
toroidal Hummel-Schroeder compactification. Then $\overline{M}$ is a
surface of general type without rational curves.
\end{theorem}

\begin{proof}

Since $\overline{M}$ admits a Riemannian metric of nonpositive
sectional curvature, the Cartan-Hadamard theorem \cite{Petersen}
implies that the universal cover of $\overline{M}$ is diffeomorphic
to the four dimensional euclidean space. Consequently,
$\overline{M}$ is aspherical and then it cannot contain rational
curves. Moreover, the second Betti number of $\overline{M}$ is even
since by construction it admits a K\"ahler metric. By the
Kodaira-Enriques classification \cite{Van de Ven} we conclude that
the Kodaira dimension of $\overline{M}$ cannot be negative.

From Proposition \ref{Milnor}, we know that the Euler number of
$\overline{M}$ is strictly positive. The minimal complex surfaces
with Kodaira dimension equal to zero and positive Euler number are
simply connected or with finite fundamental group. Since
$\pi_{1}(\overline{M})$ is infinite, the Kodaira dimension of
$\overline{M}$ is bigger or equal than one.

The fundamental group of an elliptic surface with positive Euler
number is completely understood in terms of the \emph{orbifold}
fundamental group of the base of the elliptic fibration. More
precisely, denoting by $\pi: S\longrightarrow C$ the elliptic
fibration, if $S$ has no multiple fibers then $\pi$ induces an
isomorphism $\pi_{1}(S)\simeq\pi_{1}(C)$. In the case where we allow
multiple fibers we have the isomorphism
$\pi_{1}(S)\simeq\pi^{Orb}_{1}(C)$. For these results we refer to
\cite{Morgan2}. We are now ready to show that $\overline{M}$ cannot
be an elliptic surface. When $S$ has multiple fibers, the group
$\pi_{1}(S)$ has always torsion and then it cannot be the
fundamental group of a nonpositively curved manifold. If we assume
$\pi_{1}(\overline{M})\simeq\pi_{1}(C)$, the fact that
$\pi_{1}(\overline{M})$ grows exponentially \cite{Avez} forces the
genus of the Riemann surface $C$ to be bigger or equal than two.
Since all closed geodesics in a manifold of nonpositive curvature
are essential in $\pi_{1}$, we have that the fundamental group of
the flats introduced in the compactification injects in
$\pi_{1}(\overline{M})$ and then by assumption in $\pi_{1}(C)$. By
elementary hyperbolic geometry this would imply that
$\setZ\oplus\setZ$ acts as a discrete subgroup of $\setR$, which is
clearly impossible.
\end{proof}

\begin{corollary}\label{ample}
A toroidal Hummel-Schroeder compactification has ample canonical
line bundle.
\end{corollary}

\begin{proof}

By Theorem \ref{luca} we know that $\overline{M}$ is a minimal
surface of general type without rational curves. The corollary
follows from Nakai's criterion for ampleness of divisors on surfaces
\cite{Van de Ven}. More precisely, since for a minimal surface of
general type the self-intersection of the canonical divisor is
strictly positive \cite{Van de Ven}, it suffices to show that
$K_{\overline{M}}\cdot E>0$ for any effective divisor $E$. Thus, let
$E$ be an irreducible divisor and assume $K_{\overline{M}}\cdot
E=0$. By the Hodge index theorem we must have $E\cdot E<0$. By the
adjunction formula $E$ must be isomorphic to a smooth rational curve
with self-intersection $-2$.

\end{proof}

In the arithmetic case, part of the results contained in Theorem
\ref{luca} can be derived from a theorem of Tai, see \cite{Mumford}.
Furthermore, similar results for the so-called \emph{Picard} modular
surfaces are obtained by Holzapfel in \cite{Holzapfel}.

\section{Examples}

In this section we present examples of surfaces of general type
which do not admit nonpositively curved K\"ahler metrics, but such
that their underlying smooth manifolds do admit Riemannian metrics
with nonpositive Riemannian curvature. In order to do this one needs
to understand the restrictions imposed by the nonpositive curvature
assumption on the holomorphic curvature tensor.

Thus, define
\begin{align}\notag
p=2Re(\xi),\quad q=2Re(\eta)
\end{align}
where
\begin{align}\notag
\xi=\xi^{\al}\partial_{\al},\quad\eta= \eta^{\al}\partial_{\al}.
\end{align}
In real coordinates we have
\begin{align}\notag
R(p,q,q,p)=R_{hijk}p^{h}q^{i}q^{j}p^{k}
\end{align}
while in complex terms
\begin{align}\notag
R(\xi+\overline{\xi},\eta+\overline{\eta},\eta+\overline{\eta},\xi+\overline{\xi})&=R(\xi,\overline{\eta},\eta,\overline{\xi})
+R(\xi,\overline{\eta},\overline{\eta},\xi)\\
\notag
&+R(\overline{\xi},\eta,\eta,\overline{\xi})+R(\overline{\xi},\eta,\overline{\eta},\xi).
\end{align}
We then have
\begin{align}\notag
R_{hijk}p^{h}q^{i}q^{j}p^{k}&=R_{\al\bet\ga\del}\xi^{\al}\et^{\bet}\et^{\ga}\xi^{\del}
+R_{\al\bet\gam\de}\xi^{\al}\et^{\bet}\et^{\gam}\xi^{\de}+R_{\alp\be\ga\del}\xi^{\alp}\et^{\be}\et^{\ga}\xi^{\del}\\
\notag &+R_{\alp\be\gam\de}\xi^{\alp}\et^{\be}\et^{\gam}\xi^{\de}\\
\notag
&=R_{\al\bet\ga\del}\xi^{\al}\et^{\bet}\et^{\ga}\xi^{\del}-R_{\al\bet\ga\del}\xi^{\al}\et^{\bet}\et^{\del}\xi^{\ga}
-R_{\al\bet\ga\del}\xi^{\bet}\et^{\al}\et^{\ga}\xi^{\del}\\ \notag
&+R_{\al\bet\ga\del}\xi^{\bet}\et^{\al}\et^{\del}\xi^{\ga}\\
\notag
&=R_{\al\bet\ga\del}\{\xi^{\al}\et^{\bet}\et^{\ga}\xi^{\del}-\xi^{\al}\et^{\bet}\et^{\del}\xi^{\ga}-\xi^{\bet}\et^{\al}
\et^{\ga}\xi^{\del}+\xi^{\bet}\et^{\al}\et^{\del}\xi^{\ga}\} \\
\notag
&=R_{\al\bet\ga\del}(\xi^{\al}\et^{\bet}-\et^{\al}\xi^{\bet})\overline{(\xi^{\de}\et^{\gam}-\et^{\de}\xi^{\gam})}.
\end{align}
If we assume the Riemannian sectional curvature to be nonpositive we
have
\begin{align}\notag
R_{hijk}p^{h}q^{i}q^{j}p^{k}=R_{\al\bet\ga\del}(\xi^{\al}\et^{\bet}
-\et^{\al}\xi^{\bet})\overline{(\xi^{\de}\et^{\gam}-\et^{\de}\xi^{\gam})}\leq
0.
\end{align}
In complex dimension two, the right hand side of the above equality
reduces (after some manipulations) to
\begin{align}\notag
&R_{\al\bet\ga\del}(\xi^{\al}\et^{\bet}
-\et^{\al}\xi^{\bet})\overline{(\xi^{\de}\et^{\gam}-\et^{\de}\xi^{\gam})}\\
\notag
&=R_{1\overline{1}1\overline{1}}|\xi^{1}\et^{\overline{1}}-\et^{1}\xi^{\overline{1}}|^{2}+4
Re\{R_{1\overline{1}1\overline{2}}(\xi^{1}\et^{\overline{1}}-
\et^{1}\xi^{\overline{1}})\overline{(\xi^{2}\et^{\overline{1}}-\et^{2}\xi^{\overline{1}})}\}\\
\notag
&+2R_{1\overline{1}2\overline{2}}\{|\xi^{1}\et^{\overline{2}}-\et^{1}\xi^{\overline{2}}|^{2}
+Re(\xi^{1}\et^{\overline{1}}-\et^{1}\xi^{\overline{1}})\overline{(\xi^{2}\et^{\overline{2}}-\et^{2}\xi^{\overline{2}})}\}\\
\notag
&+2Re\{R_{1\overline{2}1\overline{2}}(\xi^{1}\et^{\overline{2}}-\et^{1}\xi^{\overline{2}})
\overline{(\xi^{2}\et^{\overline{1}}-\et^{2}\xi^{\overline{1}})}\}\\
\notag
&+4Re\{R_{2\overline{2}1\overline{2}}(\xi^{2}\et^{\overline{2}}-\et^{2}\xi^{\overline{2}})
\overline{(\xi^{2}\et^{\overline{1}}-\et^{2}\xi^{\overline{1}})}\} \\
\notag
&+R_{2\overline{2}2\overline{2}}|\xi^{2}\et^{\overline{2}}-\et^{2}\xi^{\overline{2}}|^{2}.
\end{align}
Following Siu-Mostow \cite{Mostow}, we choose the ansatz
\begin{align}\notag
\xi^{1}=ia,\quad \xi^{2}=-i,\quad \et^{1}=a,\quad \et^{2}=1
\end{align}
where $a$ is a real number. We get the inequality
\begin{align}\notag
R_{1\overline{1}1\overline{1}}4a^{4}-2R_{1\overline{1}2\overline{2}}4a^{2}+R_{2\overline{2}2\overline{2}}4\leq
0.
\end{align}
Since nonpositive Riemannian sectional curvature implies nonpositive
holomorphic sectional curvature, we conclude that
\begin{align}\label{Siu}
(R_{1\overline{1}2\overline{2}})^{2}\leq
R_{1\overline{1}1\overline{1}}R_{2\overline{2}2\overline{2}}.
\end{align}

\begin{theorem}\label{nonpositive}
A toroidal Hummel-Schroeder compactification does not admit any
K\"ahler metric with nonpositive Riemannian sectional curvature.
\end{theorem}
\begin{proof}

Let us proceed by contradiction. Consider one of the elliptic
divisors added in the compactification. By the properties of
submanifolds of a K\"ahler manifold \cite{Nomizu}, we have that the
holomorphic sectional curvature tangent to the elliptic divisor has
to be zero. Let us denote such a holomorphic sectional curvature by
$R_{1\overline{1}1\overline{1}}$. By the inequality (\ref{Siu}), we
conclude that $R_{1\overline{1}2\overline{2}}=0$. As a result, the
Ricci curvature tangent to the elliptic divisor has to be zero. We
conclude that
\begin{align}\notag
K_{\overline{M}}\cdot\Sigma=\int_{\Sigma}c_{1}(K_{\overline{M}})=0,
\end{align}
which contradicts the ampleness of $K_{\overline{M}}$, see Corollary
\ref{ample}.

\end{proof}

Combining Theorems \ref{luca} and \ref{nonpositive} with Corollary
\ref{ample}, we have thus proved Theorem \ref{jubilo}.

\section{Projective-algebraicity of minimal compactifications}

Let $\overline{M}$ be a smooth toroidal compactification of a
finite-volume complex-hyperbolic surface $M$ and let $\Sigma$ denote
the compactifying divisor. The set $\Sigma$ is exceptional and it
can be blow down. The resulting complex surface, with isolated
normal singularities, it is usually referred as the minimal
compactification of $M$ \cite{Siu}. In this section we address the
problem of the projective-algebraicity of minimal compactifications
of finite-volume complex-hyperbolic surfaces. This is motivated by a
beautiful example of Hironaka, see \cite{Harthshorne} page 417,
which shows that by contracting a smooth elliptic divisor on an
algebraic surface one can obtain a nonprojective complex space. In
the arithmetic case, the projective-algebraicity of minimal
compactifications of finite-volume complex-hyperbolic surfaces it is
known by the work of Baily and Borel, see \cite{Borel}.

For completeness, we recall the theory of semi-stable curves on
algebraic surfaces and logarithmic pluricanonical maps as developed
by Sakai in \cite{Sakai}.

Let $\overline{M}$ be a smooth projective surface. Let $\Sigma$ be a
reduced divisor having simple normal crossings on $\overline{M}$.
\begin{definition}

The pair $(\overline{M},\Sigma)$ is called minimal if $\overline{M}$
does not contain an exceptional curve $E$ of the first kind such
that $E\cdot\Sigma\leq 1$.
\end{definition}

We consider the logarithmic canonical line bundle
$\mathcal{L}=K_{\overline{M}}+\Sigma$ associated to $\Sigma$. Given
any integer $k$, define
$\overline{P}_{m}=\textrm{dim}H^{0}(\overline{M},\mathcal{O}(m\mathcal{L}))$.
If $\overline{P}_{m}>0$, we define the $m$-th \emph{logarithmic
canonical} map $\Phi_{m\mathcal{L}}$ of the pair
$(\overline{M},\Sigma)$ by
\begin{align}\notag
\Phi_{m\mathcal{L}}(x)=[s_{1}(x), ...,s_{N}(x)],
\end{align}
for any $x\in\overline{M}$ and where $s_{1}, ...,s_{N}$ is a basis
for the vector space
$H^{0}(\overline{M},\mathcal{O}(m\mathcal{L}))$. At this point one
introduces the notion of logarithmic Kodaira dimension exactly as in
the closed smooth case. We denote this numerical invariant by
$\overline{k}(M)$ where $M=\overline{M}\backslash\Sigma$. We refer
to \cite{Iitaka} for further details.

\begin{definition}
A curve $\Sigma$ is semi-stable if has only normal crossings and
each smooth rational component of $\Sigma$ intersects the other
components of $\Sigma$ in more than one point.
\end{definition}

The following proposition gives a numerical criterion for a minimal
semi-stable pair $(\overline{M},\Sigma)$ to be of log-general type.
For the proof we refer to \cite{Sakai}.

\begin{proposition}\label{numerical}
Given a minimal semi-stable pair $(\overline{M},\Sigma)$ we have
that $\overline{k}(M)=2$ if and only if $\mathcal{L}$ is numerically
effective and $\mathcal{L}^{2}>0$.
\end{proposition}

We can now state one of the main results contained in \cite{Sakai}.
In what follows, we denote by $\mathcal{E}$ the set of irreducible
curves $E$ in $\overline{M}$ such that $\mathcal{L}\cdot E=0$.

\begin{others}[Sakai]\label{Bombieri}
Let $(\overline{M},\Sigma)$ be a minimal semi-stable pair of
log-general type. The map $\Phi_{m\mathcal{L}}$ is then an embedding
modulo $\mathcal{E}$ for any $m\geq5$.
\end{others}

It is then necessary to characterize the irreducible divisors in
$\mathcal{E}$. In particular, we need the following proposition.

\begin{proposition}\label{exceptional}
Let $(\overline{M},\Sigma)$ be a minimal semi-stable pair with
$\overline{k}(M)=2$. Let $E$ be an irreducible curve such that
$\mathcal{L}\cdot E=0$. If $E$ is not contained in $\Sigma$ then
$E\simeq\setC P^{1}$ and $E\cdot E=-2$.
\end{proposition}

\begin{proof}

Under these assumptions we know that $\mathcal{L}^{2}>0$. By the
Hodge index theorem
\begin{align}\notag
\mathcal{L}^{2}>0, \quad\mathcal{L}\cdot E=0 \quad \Longrightarrow
\quad E^{2}<0.
\end{align}
But now $\mathcal{L}\cdot E=0$ which implies
\begin{align}\notag
K_{\overline{M}}\cdot E=-\Sigma\cdot E\leq 0.
\end{align}
We then have $K_{\overline{M}}\cdot E=0$ if and only if $E$ does not
intersect $\Sigma$. In this case $p_{a}(E)=0$ and then $E\simeq\setC
P^{1}$ and $E^{2}=-2$. Assume now that $K_{\overline{M}}\cdot E<0$,
then $K_{\overline{M}}\cdot E=E^{2}=-1$ and therefore $E$ is an
exceptional curve of the first kind such that $E\cdot\Sigma=1$. This
contradicts the minimality of the pair $(\overline{M},\Sigma)$.
\end{proof}

We are now ready to prove the main results of this section. Let
$\setC\mathcal{H}^{2}/\Gamma$ be a finite-volume complex-hyperbolic
surface that admits a smooth toroidal compactification as in Theorem
\ref{luca}. We then have that
$\overline{\setC\mathcal{H}^{2}/\Gamma}$ is a surface of general
type with compactification divisor consisting of smooth disjoint
elliptic curves.

\begin{proposition}\label{general}
Let $\overline{M}$ be a minimal surface of general type. Let
$\Sigma$ be a reduced divisor whose irreducible components consist
of disjoint smooth elliptic curves. Then, $(\overline{M},\Sigma)$ is
a minimal semi-stable pair with $\overline{k}(M)=2$.
\end{proposition}

\begin{proof}

Recall that the canonical divisor of any minimal complex surface of
nonnegative Kodaira dimension is numerically effective \cite{Van de
Ven}. It follows that the adjoint divisor $\mathcal{L}$ is
numerically effective. An elliptic curve on a minimal surface of
general type has negative self intersection. Moreover, for a minimal
surface of general type it is known that the self-intersection of
the canonical divisor is strictly positive \cite{Van de Ven}. By the
adjunction formula
\begin{align}\notag
\mathcal{L}^{2}=K^{2}_{\overline{M}}-\Sigma^{2}>0.
\end{align}
By Proposition \ref{numerical}, we conclude that
$\overline{k}(M)=2$.

\end{proof}

Let $\setC \mathcal{H}^{2}\backslash\Gamma_{1}$ be a finite-volume
complex-hyperbolic surface which admits a smooth toroidal
compactification $\overline{M}_{1}$. Let
$(\overline{M}_{1},\Sigma_{1})$ be the associated minimal
semi-stable pair. By Theorem \ref{jubilo}, we can find a normal
subgroup of finite index $\Gamma_{2}\lhd\Gamma_{1}$ such that the
toroidal compactification $\overline{M}_{2}$ of $\setC
\mathcal{H}^{2}/\Gamma_{2}$ is a minimal surface of general type
with compactification divisor $\Sigma_{2}$. Since
\begin{align}\notag
\pi:\setC\mathcal{H}^{2}/\Gamma_{2}\longrightarrow\setC\mathcal{H}^{2}/\Gamma_{1}
\end{align}
is an unramified covering we conclude that
$\overline{k}(M_{1})=\overline{k}(M_{2})$ \cite{Iitaka}. But by
proposition \ref{general} we know that $\overline{k}(M_{2})=2$, it
follows that $(\overline{M}_{1},\Sigma_{1})$ is a minimal
semi-stable pair of log-general type. Let us summarize this argument
into a proposition.

\begin{proposition}\label{general1}
Let $(\overline{M},\Sigma)$ be a smooth pair arising as the toroidal
compactification of a finite-volume complex-hyperbolic surface. The
pair $(\overline{M},\Sigma)$ is minimal and log-general.
\end{proposition}

The following theorem is the main result of the present section.

\begin{main}\label{algebraicity}

Let $(\overline{M},\Sigma)$ be a smooth pair arising as the toroidal
compactification of a finite-volume complex-hyperbolic surface.
Then, the associated minimal compactification is projective
algebraic.
\end{main}

\begin{proof}

By Proposition \ref{general1}, the minimal pair
$(\overline{M},\Sigma)$ is log-general. By Theorem \ref{Bombieri} we
know that $\Phi_{m\mathcal{L}}$ is an embedding modulo $\mathcal{E}$
for any $m\geq5$. We clearly have that $\Sigma$ is contained in
$\mathcal{E}$. We claim that there are no other divisors in
$\mathcal{E}$. Assume the contrary. By Proposition
\ref{exceptional}, any other curve in $\mathcal{E}$ must be a smooth
rational divisor $E$ with self-intersection minus two. The
adjunction formula gives $K_{\overline{M}}\cdot E=0$ which implies
$\Sigma\cdot E=0$. This is clearly impossible. By Theorem
\ref{Bombieri} for $m\geq 5$, the map
\begin{align}\notag
\Phi_{m\mathcal{L}}: \overline{M}\longrightarrow \setC P^{N-1}
\end{align}
gives a realization of the minimal compactification as a
projective-algebraic variety.

\end{proof}

For an approach to the projective-algebraicity problem through
$L^{2}$-estimates for the $\overline{\partial}$-operator we refer to
the aforementioned paper of Mok \cite{Mok}.\\

\noindent\textbf{Acknowledgements}. I would like to thank Professor
Claude LeBrun for his constant support and for constructive comments
on the paper. I also would like to thank Professor Klaus Hulek for
some useful bibliographical suggestions and the referee for
pertinent comments on the manuscript.

\address{Mathematics Department, Duke University, Box 90320, Durham, NC 27708,
USA}

\emph{E-mail address}: \email{luca@math.duke.edu}

\end{document}